\documentclass[ 11pt]{amsart}

        \newcommand{\OO}{{\mathscr{O}}}

                \newcommand{\Y}{{\mathscr{Y}}}

\usepackage{amsmath}
\usepackage{amscd,amsthm,amssymb,amsfonts}
\usepackage{mathrsfs}
\usepackage{dsfont}
\usepackage{stmaryrd}
\usepackage{euscript}
\usepackage{expdlist}
\usepackage{enumerate}

\input xy
\xyoption{all}
\usepackage[OT2,T1]{fontenc}

\theoremstyle{plain}
\newtheorem{thm}{Theorem}[section]

\newtheorem*{thm*}{Theorem}

\newtheorem{lm}[thm]{Lemma}

\newtheorem*{cor*}{Corollary}

\newtheorem*{conj*}{Conjecture}
\newtheorem{conj}{Conjecture}

\theoremstyle{remark}
\newtheorem*{remark}{Remark}

\theoremstyle{definition}
\newtheorem*{defn*}{Definition}

\newcommand{\nc}{\newcommand}

\newcommand{\beq}{\begin{equation}}
\newcommand{\eeq}{\end{equation}}
\newcommand{\bpmx}{\begin{pmatrix}}
\newcommand{\epmx}{\end{pmatrix}}
\newcommand{\bbmx}{\begin{bmatrix}}
\newcommand{\ebmx}{\end{bmatrix}}

\newcommand{\beqcd}[1]{\begin{equation*}\label{#1}\tag{#1}}
\newcommand{\eeqcd}{\end{equation*}}

\numberwithin{equation}{section}

\def\makeop#1{\expandafter\def\csname#1\endcsname
  {\mathop{\rm #1}\nolimits}\ignorespaces}

\makeop{Hom}   \makeop{End}   \makeop{Aut}   
\makeop{Pic} \makeop{Gal}       \makeop{Div} \makeop{Lie}
\makeop{PGL}   \makeop{Corr} \makeop{PSL} \makeop{sgn} \makeop{Spf}
 \makeop{Tr} \makeop{Nr} \makeop{Fr} \makeop{disc}
\makeop{Proj} \makeop{supp} \makeop{ker}   \makeop{Im} \makeop{dom}
\makeop{coker} \makeop{Stab} \makeop{SO} \makeop{SL} \makeop{SL}
\makeop{Cl}    \makeop{cond} \makeop{Br} \makeop{inv} \makeop{rank}
\makeop{id}    \makeop{Fil} \makeop{Frac}  \makeop{GL} \makeop{SU}
\makeop{Trd}   \makeop{Sp} \makeop{Tr}    \makeop{Trd} \makeop{Res}
\makeop{ind} \makeop{depth} \makeop{Tr} \makeop{st} \makeop{Ad}
\makeop{Int} \makeop{tr}    \makeop{Sym} \makeop{can} \makeop{SO}
\makeop{torsion} \makeop{GSp} \makeop{Tor}\makeop{Ker} \makeop{rec}
\makeop{Ind} \makeop{Coker}
 \makeop{vol} \makeop{Ext} \makeop{gr} \makeop{ad}
 \makeop{Gr}\makeop{corank} \makeop{Ann}
\makeop{Hol} 
\makeop{Fitt} \makeop{Mp} \makeop{CAP}

\def\Spec{\mathrm{Spec}\,}

\DeclareMathAlphabet{\mathpzc}{OT1}{pzc}{m}{it}
\DeclareSymbolFont{cyrletters}{OT2}{wncyr}{m}{n}
\DeclareMathSymbol{\SHA}{\mathalpha}{cyrletters}{"58}

\def\makebb#1{\expandafter\def
  \csname bb#1\endcsname{{\mathbb{#1}}}\ignorespaces}
\def\makebf#1{\expandafter\def\csname bf#1\endcsname{{\bf
      #1}}\ignorespaces}
\def\makegr#1{\expandafter\def
  \csname gr#1\endcsname{{\mathfrak{#1}}}\ignorespaces}
\def\makescr#1{\expandafter\def
  \csname scr#1\endcsname{{\EuScript{#1}}}\ignorespaces}
\def\makecal#1{\expandafter\def\csname cal#1\endcsname{{\mathcal
      #1}}\ignorespaces}

\def\doLetters#1{#1A #1B #1C #1D #1E #1F #1G #1H #1I #1J #1K #1L #1M
                 #1N #1O #1P #1Q #1R #1S #1T #1U #1V #1W #1X #1Y #1Z}
\def\doletters#1{#1a #1b #1c #1d #1e #1f #1g #1h #1i #1j #1k #1l #1m
                 #1n #1o #1p #1q #1r #1s #1t #1u #1v #1w #1x #1y #1z}
\doLetters\makebb   \doLetters\makecal  \doLetters\makebf
\doLetters\makescr
\doletters\makebf   \doLetters\makegr   \doletters\makegr

\normalsize

\makeop{Ram} \makeop{Rep} \makeop{mass}

\makeop{Bl}

\def\Zp{\Z_p}

\newcommand{\Z}{\mathbf Z}
\newcommand{\Q}{\mathbf Q}

\newcommand{\C}{\mathbf C}
\newcommand{\A}{\mathbf A}    

\def\frakp{{\mathfrak p}}

\def\one{\mathbf 1} 

  \nc{\opp}{\mathrm{opp}} \nc{\ul}{\underline}

\def\ndivide{\nmid}

\def\iso{\simeq}

\def\lam{\lambda}

\usepackage[utf8x]{inputenc}

\usepackage{amsfonts}

\usepackage{amsthm}
\usepackage{amsmath,amscd}
\usepackage{amssymb}
\usepackage{graphicx}
\title[On the non-vanishing of $p$-adic heights]{On the non-vanishing of $p$-adic heights on CM abelian varieties, and the arithmetic of Katz $p$-adic $L$-functions}

\author{Ashay A. Burungale and Daniel Disegni}
\email{ashayburungale@gmail.com, daniel.disegni@gmail.com}

\def\skewhf{\vartheta}

\newcommand{\baar}{\overline}

\newtheorem{rema}[thm]{Remark}

\newcommand{\llb}{\llbracket}
\newcommand{\rrb}{\rrbracket}

\begin{document}
\maketitle

\begin{abstract} 
Let $B$ be a simple CM abelian variety over a CM field $E$, $p$ a rational prime. Suppose that $B$ has potentially ordinary reduction above $p$ and is self-dual with root number $-1$. Under some further conditions, 
we prove the generic non-vanishing of (cyclotomic) $p$-adic heights on $B$  along   anticyclotomic $\Z_{p}$-extensions of $E$. 
This provides evidence towards Schneider's conjecture on the non-vanishing of $p$-adic heights. 
For CM elliptic curves over $\Q$, the result was previously known as a consequence of work of Bertrand, Gross--Zagier and Rohrlich in the 1980s. Our proof is based on non-vanishing results for  Katz $p$-adic $L$-functions and a Gross--Zagier formula relating the latter to families of  rational points on $B$. 


\end{abstract}
\tableofcontents
\section{Introduction and statements of the main results} 

Let $B$ be an abelian variety over a number field $E$, and let $B^{\vee}$ be its dual. Let $p$ be a prime and let $L$ be a finite extension of $\Q_{p}$.
The N\'eron--Tate height pairing on $B(E')_{\Q}\times B^{\vee}(E')_{\Q}$, where $E'$ is a finite extension of $E$, admits a $p$-adic analogue (see e.g. \cite{Nht})
\begin{align}\label{HP}
B(E')_{\Q}\times B^{\vee}(E')_{\Q}\to L
\end{align}
depending   on the choices of a  homomorphism (``$p$-adic logarithm'')  $\ell\colon E^{\times}\backslash E_{\A^{\infty}}^{\times}\to L$ and on splittings of the Hodge filtration on $H^{1}_{\rm dR}(B/E_{v})$ at the primes $v\vert p$; in the potentially ordinary case under consideration in this paper, there are canonical  choices (the ``unit root'' subspaces) for the Hodge splittings. While it is a classical result that the N\'eron--Tate height pairing is non-degenerate, the pairing \eqref{HP} can vanish for some choices of $\ell$. Suppose however that $\ell=\ell_{\Q}\circ N_{E/\Q}$ with $\ell_{\Q}$ a $p$-adic logarithm of $\Q$ such that $\ell_{\Q}|_{1+p\Z_{p}}
$ is nontrivial
(we call such $\ell$ a \emph{cyclotomic} logarithm).
Then it is conjectured \cite{Sc2} that \eqref{HP} is  non-vanishing.
 This  long-standing  conjecture is only known in a few special cases: for CM elliptic curves, thanks to Bertrand \cite{Be}, and also  for elliptic curves over $\Q$ at supersingular primes.\footnote{This was observed in \cite{BPR}; see \cite[\S 4.5]{kobayashi}  for the comparison between the  definition  of the height pairing used in \cite{BPR} and the `standard' definition of Zarhin and Nekov{\'a}{\v{r} } \cite{Nht}.}
The stronger statement that \eqref{HP} is non-degenerate is also conjectured to be true; this is not known in any cases of rank higher than~$1$. The non-degeneracy conjecture has arithmetic consequences: it allows to formally deduce the $p$-adic Birch and Swinnerton-Dyer conjecture from the Iwasawa main conjecture for $B$.\footnote{See \cite{Sc2}. For further applications to the \emph{classical} Birch and Swinnerton-Dyer formula, see \cite{PR, kobayashi, Di1}.}


The pairing \eqref{HP} is  equivariant for the actions of $\Gal(E'/E)$ and of $K:=\End^{0}(B)$ on both sides; for our purposes we can then assume that $B$ is simple,  that the coefficient field $L$ is sufficiently large and has the structure of  a $K\otimes \Q_{p}$-algebra,  and then decompose the pairing  into isotypic components
\begin{align}\label{HPchi}
B(\chi) \otimes_{L} B^{\vee}(\chi^{-1})\to L
\end{align}
for the $\Gal(E'/E)$-action. Here and in the rest of the paper, if  $R$ is a  $K\otimes \Q_{p}$-algebra and  $\chi$ is an $R^{\times}$-valued   character of $\Gal({E}^{\rm ab}/E)$, we define 
$$B(\chi):=(B({E}^{\rm ab})\otimes_{K\otimes\Q_{p}}  R(\chi))^{\Gal({E}^{\rm ab}/E)}, $$ 
where $R({\chi})$ is a rank-$1$ $R$-module with Galois action by $\chi$.

\medskip 

The most significant result of this  paper is the proof that, under some assumptions,  the non-vanishing conjecture  for \eqref{HPchi} is true ``generically'' when $B$ is a  $p$-ordinary CM abelian variety over a CM field $E$ and $\chi$  varies  among  anticyclotomic characters of   $E$  unramified outside $p$. 

The method of proof, different  from that of previous results on this topic, is  automorphic.  (In particular, the approach does not involve transcendence arguments.)
It combines two ingredients. The first is a pair of nonvanishing results for Katz $p$-adic $L$-functions due to Hida, Hsieh, and the first author (in turn  relying on  Chai's results  on  Hecke-stable subvarieties of a mod $p$ Shimura variety \cite{Ch1, Ch2, Ch3}).  The second ingredient is a  Gross--Zagier  formula  relating  derivatives of Katz $p$-adic $L$-functions to families of rational points  on CM abelian varieties. We deduce this formula from work of 
the second author, by an argument employed by Bertolini--Darmon--Prasanna \cite{BDP} in a different context.

In the rest of this section we describe the main results in  more detail.

\subsection{Non--vanishing of $p$-adic heights} 
Let $B/E$ be a simple CM abelian variety over a CM field $E$, i.e. $K:=\End^{0}(B)$ is a CM field of degree $[K:\Q]=2\dim B$. Let $F$ be the maximal totally real subfield of $E$,  and let $\eta=\eta_{E/F}$ be the associated character of $F^{\times}\backslash F_{\A}^{\times}$.

\subsubsection*{Assumptions} The abelian variety  $B$ is associated with a Hecke character 
$$\lam=(\lambda^{\tau})_{(\tau\colon K \hookrightarrow \C)}\colon E^{\times}\backslash E_{\A}^{\times}\to (K\otimes \C)^{\times} $$ (\cite[\S 19]{Sh3}). 
Suppose that $\lambda$ satisfies the  condition
\begin{align}\label{cc}
\lambda|_{F_{\A}^{\times}} =\eta   |\cdot|_{\A}^{-1};
\end{align}
this holds in particular whenever $B$  arises as the base-change of a real-multiplication abelian variety $A/F$ \cite[Theorem 20.15]{Sh3}.  It implies that the 
 for each $\tau$ there is a functional equation with sign 
 $w(\lam):=\varepsilon(1, \lambda^{\tau})\in\{\pm 1\}$
  (independent of $\tau$) relating $L(s,\lam^{\tau})$ to $L(2-s, \lam^{\tau})$. 
 We will assume that 
\begin{align}\label{root num}
w(\lam)=-1.
\end{align}

Finally,
we assume that 
\begin{align}\label{ord}
\text{$B$ has potentially \emph{ordinary} reduction at every prime of $E$ above $p$.}
\end{align}



%
%

\subsubsection*{Anticyclotomic regulators} Let $E_{\infty}^{-}$ (respectively  $E_{\infty}^{+}$) be the anticyclotomic $\Zp^{[F:\Q]}$-extension (respectively cyclotomic $\Zp$-extension) of $E$ and, for a prime $\wp $ of $F$ above $p$, let
$E_{\wp,\infty}^{-} \subset E_{\infty}^{-}$ 
be the $\wp$-anticyclotomic subextension, i.e. the maximal subextension unramified outside the primes above $\wp$ in $E$; finally, let  $E_{\infty}:=E^{-}_{\infty}E^{+}_{\infty}$ and $E_{{\wp},\infty}=E_{\wp,\infty}^{-}E_{\infty}^{+}$. 
If $\bullet$ is any combination of subscripts $\emptyset$, $\wp$ and superscripts $\emptyset$, $+$, $-$ (we convene that `$\emptyset$'  denotes no symbol),  the corresponding infinite Galois group is 
$$\Gamma_{\bullet}:=\Gal(E_{\bullet, \infty}/E).$$

%
%

Let $L$ be a
 finite extension of a $p$-adic completion $K_{w}$ of $K$.
If $\bullet$ is any set of sub- and superscripts as above,
we let
$$\Lambda_{\bullet}:= \mathscr{O}_{L}\llbracket \Gamma_{\bullet}\rrbracket\otimes L, \qquad \Y_{\bullet}:=\Spec \Lambda_{\bullet}.$$
When we want to emphasise the role of a specific choice of $L$, we will write $\Lambda_{\bullet, L}$, $\Y_{\bullet, L}$.

For $\circ=\wp, \emptyset$ we let 
$$\chi_{{\rm univ}, \circ}\colon \Gamma_{\circ}^{-}\to \Lambda_{\circ}^{-, \times}$$
 be the tautological anticyclotomic  character. 
%
%
%
We then   have a $\Lambda^{-}_{\circ}$-module 
$$B(\chi_{\rm univ, \circ}):=(B(\overline{E})\otimes \Lambda_{\circ}^{-}(\chi_{\rm univ, \circ}))^{\Gal(\overline{E}/E)},$$
 whose specialisation at any finite order character $\chi\in \mathscr Y^{-}$ is $B(\chi)$, 
 and height pairings
(\cite[\S 2.3]{PR2}, see also \cite[\S11]{nek-selmer})
\begin{align}\label{bigHP}
B(\chi_{\rm univ , \circ})\otimes_{ \Lambda_{\circ}^{_{}}} B^{\vee}(\chi_{\rm univ, \circ}^{-1}) &\to \Lambda_{\circ}^{-}
\end{align}
associated with choices of a $p$-adic logarithm $\ell$ and of Hodge splittings. We suppose that  $\ell\colon {\Gal(E^{\rm ab}/E})\to L$ is the cyclotomic logarithm and that the Hodge splittings are given by the unit root subspaces. 

%
%



\begin{thm}\label{MT}
Let $B$ be a simple CM abelian variety over the CM field $E$ with associated Hecke character $\lam$ satisfying   \eqref{cc}, \eqref{root num}, and \eqref{ord}; suppose that the extension $E/F$ is ramified.
Let $p$ be a rational prime, and suppose that $p \ndivide 2 D_{F} h_E^-$,
where $h_E^-=h_E/h_F$ is the relative class number of $E/F$ and $D_{F}$ is the absolute discriminant of $F$. 
Let   ${\wp}$  be a prime of $F$ above $p$, and let $L\supset K_{w}\supset K$ be as above.

Then for almost all finite-order characters  $\chi$ of $\Gamma^{-}_{\wp}$, the pairing \eqref{HPchi} 
$$B(\chi) \otimes_{L} B^{\vee}(\chi^{-1})\to L$$
 is non-vanishing. Equivalently, the paring \eqref{bigHP} for $\circ=\wp$  (hence also for $\circ=\emptyset$) is nonzero.
\end{thm}
Here ``almost all'' means that the set of finite-order characters $\chi\in \mathscr{Y}_{\wp}^{-}$ which fail to satisfy the conclusion of the theorem is not Zariski dense in $\mathscr{Y}_{\wp}^{-}$. When $\dim \mathscr{Y}_{\wp}^{-}=1$ (e.g. $F=\Q$), this is equivalent to such set of exceptions being finite.


\subsection{Gross--Zagier formula for the Katz $p$-adic $L$-function} As recalled in \S\ref{PLF}, under our conditions and assuming that the extension $L\supset K_{w}$ splits $E$, 
  the character $\lam$ (more precisely,  its $w$-adic avatar)  has a $p$-adic  CM type $\Sigma_{E}\subset \Hom(E, L)$ of $E$. We also identify $\Sigma_{E}$ with  (i) a choice, for each prime $\wp\vert p$ of $F$, of one among the two primes of $E$ above $\wp$, and (ii)  an element of $\Z[\Hom(E, L)]$.  To  the CM type $\Sigma_{E}$ is attached 
the Katz $p$-adic $L$-function 
$$L_{\Sigma_{E}}\in
\Lambda.$$
 It interpolates the values $L(0, \lam'^{-1}) $ for characters $\lam'$ whose infinity type lies  in a certain region  of $\Z[\Hom(E, \C_{p})]$; this region is uniquely determined by $\Sigma_{E}$ and  contains in particular the infinity  type $\Sigma_{E}$. 

Let $\lam^{*}(x):=\lam(x^{c})$, where $c$  denotes the complex conjugation of $E/F$. 
The root number assumption \eqref{root num}  implies that
 $L_{\Sigma_{E}}(\lam)= L_{\Sigma_{E}c}(\lam^{*})=L(0, \lam^{*-1})=L(0, \lam^{-1})=L(1, \lam)=L(B,1)=0$, and more generally that the function $L_{\Sigma_{E}c, \lam^{*}}\colon \chi'\mapsto L_{\Sigma_{E}c}(\lam^{*}\chi')$ vanishes along $\mathscr Y^{-}\subset \mathscr Y$. We may then consider
  the  cyclotomic derivative  
$$L'_{\Sigma_{E}c, \lam^{*}}\colon\chi\mapsto   {{\rm d}\over{\rm d } s}  L_{\Sigma_{E} c}( \lam^{*}\chi \cdot \chi_{\rm cyc}^{s})|_{s=0},$$
 for $\chi\in \mathscr{Y}^{-}$. (Here $\chi_{\rm cyc}$ is the $p$-adic cyclotomic character of $E_{\A}^{\times}$.)

For $\circ=\emptyset, \wp$, let  ${\mathscr{K}_{\circ}^{-}}$ be the field of fractions of $ \Lambda^{-}_{\circ}$
 and let $B(\chi_{\rm univ, \circ})_{\mathscr{K_{\circ}^{-}}}:=B(\chi_{\rm univ, \circ})\otimes_{\Lambda_{\circ}^{-}}{\mathscr{K}_{\circ}^{-}}$; similarly for  $B^{\vee}(\chi_{\rm univ}^{-1})_{\mathscr{K_{\circ}^{-}}}$.

\begin{thm}\label{MT2}  Let $\circ=\emptyset $ or $\circ =\wp$. Under   the assumptions of Theorem \ref{MT}, 
there is a `pair of  points' 
$$\mathscr P\otimes\mathscr P^{\vee}\in B(\chi^{-1}_{\rm univ, \circ})_{\mathscr{K_{\circ}^{-}}}\otimes B(\chi_{\rm univ, \circ})_{\mathscr{K_{\circ}^{-}}}$$
satisfying 
$$\langle \mathscr P, \mathscr P^{\vee}\rangle_{\circ} =  L_{\Sigma_{E}c, \lam^{*}}'|_{\Y_{\circ}^{-}}$$
in $\mathscr{K}_{\circ}^{-}$, where $\langle\, ,\, \rangle_{\circ}$ is the height pairing \eqref{bigHP},
and we identify $\Gamma^{+}\cong \Z_{p}$ via the cyclotomic logarithm.
\end{thm}

The  construction of the points depends on some choices, analogously to how the construction of rational points on an elliptic curve over $\Q$ of analytic rank one depends on the choice of an auxiliary imaginary quadratic field. Like in that situation, it comes from Heegner points and relies on a non-vanishing result for $L$-functions -- in this case, the results of Hida and Hsieh \cite{Hi3, Hs1} for anticyclotomic Katz $p$-adic $L$-functions. Nevertheless the auxiliary setup does not seem to be explored in regard to the cyclotomic derivative. 

The following conjecture, which can be regarded as analogous to the results of Kolyvagin,  would imply that the ambiguity is rather mild.
\begin{conj} Let $\circ=\emptyset $ or $\circ=\wp$ for a prime $\frakp$ of $F$. The $\mathscr{K}_{\circ}^{-}$-vector spaces $B(\chi^{-1}_{\rm univ, \circ})_{\mathscr{K}_{\circ}^{-}}$, $B^{\vee}(\chi_{\rm univ, \circ})_{\mathscr{K}_{\circ}^{-}}$ have dimension one.  
\end{conj} 
When  $E$ is an imaginary quadratic field and $B$ is the  base-change of an elliptic curve over $\Q$, the conjecture is part of the main result of Agboola and Howard in \cite{AH}. 

\begin{remark}
It is natural to  wonder about  the arithmetic significance of the values 
\beq \label{L values} 
L_{\Sigma_{E}'}(\lam^{*}) \eeq
for CM types $\Sigma_{E}'$ such that 
$$\delta(\Sigma_{E}'):=|\Sigma_{E}'\cap \Sigma_{E}|\geq 1.$$
 We would like to suggest  that  if   $  \delta (\Sigma_{E}')\leq r:= {\rm ord}_{s=1} L(B, s)$, the cyclotomic order of vanishing of $L_{\Sigma_{E}'}$ at $\lam^{*}$ (that is, the smallest $k$ such that the cyclotomic derivative $L_{\Sigma_{E}'}^{(k)}(\lam^{*})\neq 0 $) should be 
\beq \label{conj a la rubin}
{\rm ord}_{\rm cyc} L_{\Sigma_{E}',\lam^{*}} \stackrel{?}{=}  r-\delta(\Sigma_{E'})
\eeq
  and that there should be 
 an explicit formula relating
 $$L_{\Sigma_{E}'}^{(r-\delta(\Sigma_{E'}))} (\lam^{*})$$ to a $p$-adic regulator.  When $[E:\Q]$ is a quadratic field and $\lambda$ comes from an elliptic curve, \eqref{conj a la rubin} was conjectured by Rubin \cite{rubin-conj} together with a precise formula, and proved by himself when $r\leq 1$.
 
Theorem \ref{MT2}  provides evidence for the general case  of \eqref{conj a la rubin} in  one of the cases  with  $r\leq 1$, whereas the other such case is treated, when $[E:\Q]=2$, by Rubin's formula as generalised by Bertolini--Darmon--Prasanna \cite{BDP}.\footnote{We hope to present a generalisation of this formula in a sequel to the present paper.} The particular interest of the case of $[E:\Q]>2$ lies of course in the possibility of having $\delta(\Sigma_{E}')\geq 2$:  our speculation 
is also  inspired from  the  recent work of Darmon--Rotger \cite{DR} on $p$-adic $L$-functions related to certain   Mordell--Weil groups of  rank~$2$. 
 \end{remark}
 

\subsection{An arithmetic application} For an elliptic curve $E_{/\Q}$ and a prime $p$, let us  refer to 
 the implication 
$$
\corank_{\Z_{p}}{\rm Sel}_{p^{\infty}}(E_{/\Q})=1 \implies {\rm ord}_{s=1}L(s,E_{/\Q})=1
$$ 
as the ``$p$-converse theorem''  (to the one of Gross--Zagier, Kolyvagin and Rubin).
In \cite{BuTi}, the authors establish the $p$-converse theorem in the case of $p$-ordinary CM elliptic curves  (complementing the earlier works \cite{Sk, Zh}). The approach crucially relies on the auxiliary setup introduced in the proof of the main results here, and also on the main results themselves (Theorem \ref{MT} and Theorem \ref{MT2}). 
\subsection{Context and strategy of proof}
When $E$ is an imaginary quadratic field and $B=A_{E}$ for a  CM elliptic curve $A/\Q$, Theorem \ref{MT} is a  consequence (see Rubin \cite{rubin})  
of the aforementioned result of Bertrand together with Mazur's conjecture on the generic non-vanishing of Heegner points along anticyclotomic extensions (which in that case is proved, via the Gross--Zagier formula, by the generic non-vanishing of derivatives of $L$-functions established by Rohrlich). 
Our method is rather different (although  not  distant in spirit): we first deduce the formula of Theorem of \ref{MT2} from the $p$-adic Gross--Zagier formula of the second author \cite{Di}.  
Theorem \ref{MT}, or rather its more precise version Theorem \ref{MT-text} below,  is then a consequence of 
Theorem \ref{MT2}  and non-vanishing results of the first author \cite{Bu} (or their refinement in  \cite{BuHi}) for the derivatives of  Katz $p$-adic $L$-functions.
As a corollary we recover Mazur's conjecture, which in our case was proved by Aflalo--Nekov{\'a}{\v{r} }\cite{AN} as a generalisation of work of Cornut and Vatsal. Note further that our method  would readily adapt  to cover the case of (generalised) Heegner cycles upon availability of a suitably general $p$-adic Gross--Zagier formula for them.\footnote{This case would be especially interesting as the  archimedean heights are not known to be well-defined in that context.} The second author expects to present such a formula as part of a forthcoming version of \cite{Di-u}.

\medskip

A parallel  approach is followed by the first author in a series of works establishing, without assumptions of complex multiplication, the generic non-vanishing of  Heegner points and cycles, or more precisely of (the reduction modulo $p$ of) their images under the Abel--Jacobi map.\footnote{Note that, as $p$-adic heights factor through the $p$-adic Abel--Jacobi map,  their  non-vanishing is a finer statement.} The    strategy to prove Theorem \ref{MT2} 
is inspired from the  proof of  Rubin's formula  in  \cite{BDP}.
As in    \cite{BDP}, we can remark that we have established a result for a motive attached to the group ${\rm U}(1)$ by making use of $p$-adic  $L$-functions for  ${\rm U}(1)\times {\rm U}(2)$. Readers with a generous attitude towards  mathematical induction might find in this  a  good omen  for  future progress. 
%
\subsubsection*{Acknowledgments} 
We would like to thank  D. Bertrand, K. B\"uy\"ukboduk, F. Castella,   M. Chida,  H. Darmon,  H. Hida, B. Howard,  M.-L. Hsieh, M. Kakde, J. Lin, J.  Nekov\'{a}\v{r}, D. Prasad, D. Rohrlich,   K. Rubin,  and Y. Tian for useful conversations, correspondence  or suggestions.  During part of the writing of this paper, the second author was funded by a public grant from the Fondation Math\'ematique Jacques Hadamard. 

\section{Proofs}

\subsection{CM theory }\label{sec: cm} We start by reviewing some basic results in the theory of Complex Multiplication. The classical reference is \cite{Sh3} (see especially \S18-20); see also \cite[\S 2.5]{CCO}.

Let $B$ be an abelian variety of dimension $d$ over a field $E$, such that ${\rm End}^{0}(B)=K$ is a CM field of degree $2d$. Denote by   $M\subset K$  the maximal totally real subfield.  The action of $K$ on ${\rm Lie}\, B$ gives, after base-change from $E$ to an extension $\iota\colon E\hookrightarrow C$ which splits $K$, a CM type $(K, \Sigma)$ over $C$, namely $\Sigma=\Sigma(B, \iota)$ is  a set of representatives for the action of $\Gal(K/M)$ on $\Hom (K, C)$. Finally, to the CM type $(K, \Sigma)$ we can associate its reflex CM type $(K^{*}, \Sigma^{*})$; the reflex field $K^{*}=K^{*}(\iota)$ (which depends on $\iota$) comes as a subfield $K^{*}\subset E$. The set $\Sigma_{E}:={\rm Inf}_{E/K^{*}}\Sigma^{*}\subset \Hom (E, C)$, consisting of those embeddings whose restriction to $K^{*}$ belongs to $\Sigma^{*}$, is a CM type of $E$. Finally, the CM type $\Sigma(\iota)$ gives rise to a homomorphism $N_{\Sigma(\iota)^{*}}\colon K^{*}(\iota)^{\times}\to K^{\times}$ called the \emph{reflex norm}. The homomorphism
\begin{align}\label{ref orm}
N_{\Sigma_{E}}:= N_{\Sigma(\iota)^{*}}\circ N_{E/K^{*}(\iota)}\colon E^{\times}\to K^{\times}
\end{align}
is independent of choices.

A CM type $\Sigma'$ of $K'$ with values in an extension $C$ of $\Q_{p}$ splitting $E$ is said to be  a \emph{$p$-adic CM type}\footnote{In some of the literature this is called a \emph{$p$-ordinary CM type}.} if its elements induce pairwise distinct $p$-adic places of $K'$.
This condition can only be satisfied if all primes $\frakp'^{+}$  of  $K'^{+}$  (the maximal totally real subfield of $K'$) above $p$ split in $K'$. We may and will identify a $p$-adic CM type with a set of primes $\frakp'$ of $K'$ containing exactly one prime above each $\frakp'^{+}\vert p$ of $K'^{+}$.

\begin{lm}\label{K*} Suppose that $B$ has potentially \emph{ordinary}  reduction at all primes of $E$ above $p$. 
Then:
\begin{enumerate}
\item for each  embedding $\iota_{p}\colon E\hookrightarrow \overline{\Q}_{p}$,  the set $\Sigma(B, \iota_{p})$ is a $p$-adic CM type of $K$;
\item the prime $p$ is totally split in $K^{*}$; 
\item  for each $\iota_{p}\colon E\hookrightarrow \overline{\Q}_{p}$,  the set $\Sigma^{*}(\iota_{p}):=\Sigma(B, \iota_{p})^{*}$ is a $p$-adic CM type of $K^{*}$.
\end{enumerate}
\end{lm}
\begin{proof} Part (1), which is in fact equivalent to the hypothesis of the lemma, can be checked after base-change from $E$ to a finite extension over which $B$ acquires good reduction.  There it becomes  a well-known immediate consequence of the  Shimura--Taniyama formula \cite[(2.1.4.1)]{CCO}. Part (1) implies part (2) by \cite[Proposition 7.1]{kashio yoshida}.  Part (3) is implied by part (2).
\end{proof}

\medskip

The main theorem of Complex Multiplication attaches to $B$ a character 
$$\lambda\colon  E_{\A^{\infty}}^{\times}\to K^{\times}$$
such that \begin{gather}\label{cass1}
\text{for all $\tau\colon E\hookrightarrow \C$,}\qquad 
\lam^{\tau}:=\tau\circ \lambda \cdot (N_{\Sigma_{E}, \infty}^{\tau})^{-1}\colon E_{\A}^{\times}\to \C^{\times}\qquad \text{satisfies $\lam^{\tau}_{|E^{\times}}=1$},
\end{gather}
\begin{gather}\label{cass2}
 \lambda(x)\lam(x)^{\rho}=|x|_{\A^{\infty}}^{-1}\qquad \text{for all } x \in E^{\times}_{\A^{\infty}}.
\end{gather}
Here $\rho$ is the complex conjugation in $K$ and $N_{\Sigma_{E}, \infty}^{\tau}\colon E_{\infty}^{\times}\to K_{\infty}^{\times}\to K_{\tau}^{\times}$ is the continuous extension of   $N_{\Sigma_{E}}$.
We say that $\lam^{\tau}$ is an  an algebraic Hecke character of infinity type $\Sigma_{E}^{\tau}$, where $\Sigma_{E}^{\tau}\subset \Hom(E, \C)$ is defined by  $N_{\Sigma_{E,\infty}}^{\tau}(x)=\prod_{\iota\in \Sigma_{E}^{\tau}}\iota(x)$ for all $x\in E_{\infty}^{\times}$.

%
  
   The $L$-function of $B$ is  
\begin{align}\label{Blam}
L(B, s)= L(s, \lambda):= (L(s,\lambda^{\tau}))_{\tau\in \Hom(K, \C)}\in K\otimes \C\cong \C^{\Hom(K, \C)};
\end{align}
it satisfies a functional equation  with centre  at $s=1$.

Suppose now that $B=A_{E}$ for an abelian variety $A/F$ with $\End^{0}(A)=M$ (the maximal totally real subfield of $K$). Then $E=K^{*}F$ \cite[Remark 20.5]{Sh3}.

Suppose conversely that $\lam$ is a Hecke character of $E$ satisfying the conditions \eqref{cass1}, \eqref{cass2} for a CM type $(K, \Sigma)$. Then by  a theorem of Casselman (e.g.  \cite[Theorem 2.5.2]{CCO}), there is an abelian variety $B=B_{\lam}/E$, unique up to $E$-isogeny,  satisfying \eqref{Blam}. The abelian variety $B$ is simple if and only if the CM type $(K, \Sigma)$ is not induced from a CM type of a subfield of $K$.

\subsection{Theta lifts of Hecke characters}\label{sec: theta lift}
Let $\overline{\Q}$ denote an algebraic closure of $K$,  let $\chi_{0}\colon E^{\times}\backslash E_{\A}^{\times}\to \overline{\Q}^{\times}$ be  finite order character, and let $\psi$ be a Hecke character of $E$ with the same CM type as the character $\lam$ from the Introduction. We suppose that  
\begin{align}\label{cchar}
\chi_{0}|_{F_{\A}^{\times}}=\omega:=\eta\cdot \psi|_{F_{\A}^{\times}}\cdot |\,\,|_{\A_{F}}.
\end{align}
Let $K'\subset \overline{\Q}$ be a CM extension of $K$  containing the values of $\chi_{0}$ and $\psi|_{E^{\times}_{\A^{\infty}}}$.

\subsubsection*{The abelian variety associated with $\psi$}  
By construction, $\psi$ satisfies conditions  \eqref{cass1}, \eqref{cass2} for the CM type $(K',{\rm Inf}_{K'/K}\Sigma)$, and in particular it is associated with an abelian variety $B_{\psi, K'}/E$ of dimension $[K':\Q]/2$, which is uniquely determined up to $E$-isogeny,  has CM by $K'$, and is $K$-linearly isogenous to a sum of copies of a simple CM abelian variety. (Note that $B_{\psi, K'}$ depends on the choice of $K'$: if $K'\subset K''$ is a finite extension of CM fields of degree $d'$, then $B_{\psi, K''}\sim B_{\psi, K'}^{\oplus d'}$.) We denote by $B_{\psi}^{\natural}$ any one of the simple $E$-isogeny factors of $B_{\psi, K'}$.

On the other hand, the theta correspondence attaches to $\psi$ a $\Gal(M/\Q)$-conjugacy class (for some totally real field $M\subset K'$)  of cuspidal  automorphic representations $\sigma=(\sigma^{\tau_{M}})_{(\tau_{M}\in \Hom(M, \C))}$ of ${\rm Res}_{F/\Q}{\bf GL}_{2}$ of parallel weight $2$; namely $\sigma^{\tau_{M}}=\theta(\psi^{\tau'})$ if $\tau'\colon K'\hookrightarrow \C$ satisfies $\tau'|_{M}=\tau_{M}$.\footnote{In the terminology of \cite{YZZ, Di}, $\sigma$ is an \emph{$M$-rational} representation.} 

Let $A:=A_{\sigma}/F$ be the simple abelian variety associated with $\sigma=\theta(\psi)$, which  is determined uniquely up to $F$-isogeny. 
Suppose that 
\begin{align}\label{root n}
\varepsilon( 1/2, \sigma_{E}\otimes\chi_{0}^{-1})= -1
\end{align}
 As discussed in the introductions to \cite{YZZ} or \cite{Di}, the condition \eqref{root n}
guarantees that $A$ can be found as an isogeny factor of the Jacobian of a Shimura curve over $F$; its endomorphism algebra   $\End^{0}A$ is a totally real field of dimension $d=\dim A$, which can be identified with $M$. 

 For any embedding $\tau_{M}\colon M\hookrightarrow \C$, we have $L(s, \sigma^{\tau_{M}})=L(s, \psi^{\tau'})$ for $\tau'\colon K \hookrightarrow \C$ such that $\tau'|_{M}=\tau_{M}$. 
The following consequences of this identity are proved in \cite{Hi} (note that the CM type of $\psi^{\tau'}$ is $\Sigma_{E}$): (1) $A$ acquires complex multiplication by $K$ over some finite extension of $F$, and 
in fact, as remarked before, a minimal such extension is $K^{*}F=E$; (2) $A_{E}$ is isogenous to a sum of  copies of the abelian variety $B_{\psi}^{\natural}$ defined above.  

As $M$ is contained in the maximal real subfield of $K'$, the dimension $[M:\Q]$ of $A_{E}$ divides the dimension $[K':\Q]/2$ of $B_{\psi,K'}$; hence in fact  $B_{\psi, K'}$ is $K$-linearly isogenous to a sum of copies of $A_{E}$, i.e. for some $r\geq 1$, 
\beq\label{bpsi}
B_{\psi, K'}\sim A_{E}^{\oplus r}.\eeq


\subsection{Rankin--Selberg $p$-adic $L$-function}  Moving to a more general context for this subsection only, let   $F$ be a totally real field, $E$ a CM quadratic extension of $F$. 
Let $M$ be a number field and let $\sigma$ be an $M$-rational automorphic representation of ${\rm Res}_{F/\Q}{\bf GL}_{2}$ of parallel weight $2$, with central character $\omega$. Let $M_{v}$ be a $p$-adic completion of $M$, let $L$ be a finite extension of $M_{v}$, and let 
$$\chi_{0}\colon  E^{\times}\backslash E_{\A}^{\times}\to L^{\times}$$
be a finite-order character  satisfying $\chi_{0}|_{F_{\A}^{\times}}=\omega$.
(We will later specialise to the situation 
 $\sigma=\theta(\psi)$ considered in \S\ref{sec: theta lift}).    We recall the definition of a $p$-adic Rankin--Selberg $L$-function on $\mathscr{Y}_{L}$ attached to  the base-change  $\sigma_{E}$ of $\sigma$ to $E$ twisted by $\chi_{0}^{-1}$.


Given a place $\wp\vert p$ of $F$, we say that $\sigma$ is \emph{nearly ordinary at $\wp$} with unit characters $\alpha_{\wp}$ if, after possibly enlarging $L$, there exist  
   characters $\alpha_{\wp}\colon F_{\wp}^{\times}\to \mathscr{O}_{L}^{\times}$, such that
$\sigma_{\wp}\otimes L$ is either special $\alpha_{\wp}\cdot{\rm St}$ with character $\alpha_{\wp}$, or irreducible principal series ${\rm Ind}(|\cdot|_{\wp}\alpha_{\wp},\beta_{\wp})$ (un-normalised induction) for some other character $\beta_{\wp}$. 


   \begin{thm}  \label{stdLp}  Suppose  that for all $\wp\vert p$, $\sigma$ is nearly ordinary at $\wp$.  Then there is a function 
$$L_{p}(\sigma_{E}\otimes \chi_{0}^{-1}) \in \Lambda$$
characterised by the interpolation property
$$L_{p}(\sigma_{E}\otimes \chi_{0}^{-1}) (\chi') =   e_{p}(\sigma_{E}^{\iota}\otimes(\chi_{0}\chi')^{\iota, -1})
\cdot{ L^{(p)}(1/2,\sigma_{E}^{\iota}\otimes(\chi_{0}\chi')^{\iota, -1})  \over \Omega_{\sigma}^{\iota}} $$
for all sufficiently $p$-ramified\footnote{See the proof for the precise meaning.} finite order characters
$\chi'\colon \Gamma \to \baar{\Q}^{\times}$ and $\iota\colon \baar{\Q}\hookrightarrow\C$.
Here $\Omega_{\sigma}^{\iota}:= L(1, \sigma^{\iota}, {\rm ad})$ and $e_{p}(\sigma_{E}^{\iota}\otimes (\chi_{0}\chi')^{\iota,-1})=\prod_{\wp \vert p}e_{\wp}(\sigma_{E}^{\iota}\otimes (\chi_{0}\chi')^{\iota,-1})$ with 
\beq\label{eps-fact}
e_{\wp}(\sigma_{E}^{\iota}\otimes (\chi_{0}\chi')^{\iota,-1})=
\varepsilon(0, \iota(\alpha_{\wp}^{-1}\chi_{0,\frakp}\chi'_{\frakp}))\cdot
\varepsilon(0,\iota( \alpha_{\wp}^{-1}\chi_{0,\frakp^{c}}\chi'_{\frakp^{c}})).
\eeq
\end{thm}
\begin{proof}
   This is essentially \cite[Theorem A]{Di}.
Our $L_{p}(\sigma_{E}\otimes\chi_{0}^{-1})$ differs from the one of \emph{loc. cit.} by the involution $\chi'\mapsto \chi'^{-1}$, the shift $\chi_{0}^{-1}$, and some algebraic constants. Regarding the interpolation factors, note that if $\chi'_{\frakp}$, $\chi'_{\frakp^{c}}$ are sufficiently ramified then all local $L$-values in the interpolation formula from \cite{Di} are equal to $1$. The relation between  the Gau\ss\  sums used in \cite{Di} and our epsilon factors \eqref{eps-fact} follows from e.g. \cite[(23.6.2)]{BH}.
\end{proof}

\subsection{Katz $p$-adic $L$-function}\label{PLF}   From now until the end of this paper, let  $K$, $E$, $\Sigma_{E}$ be  as in \S\ref{sec: cm}.  Fix a $p$-adic  place $w$ of $K$, and denote by  $v$ be the induced place of the maximal real subfield $M\subset K$. We  let  $\overline{\Q}\subset\overline{\Q}_{p}\supset K_{w}$ be algebraic closures of $K$ and of $K_{w}$,  and let   $\C_{p}$ be the completion of $\overline{\Q}_{p}$.  We also let $L$ be a sufficiently large   finite extension of $K_{w}$ inside $\overline{\Q}_{p}$ as in the introduction.

\subsubsection{$p$-adic CM type associated with $w$}
Let
 $ N_{\Sigma_{E}, p}^{(w)}\colon E_{p}^{\times}\to K_{p}^{\times}\to K_{w}^{\times}$ be the continuous extension of $N_{\Sigma_{E}}$. We let 
$$\Sigma_{E}^{(w)}:=\{\frakp|p \text{ prime of } E\, : \, |\cdot|_{w}\circ N_{\Sigma_{E},p}^{(w)}|_{E^{\times}_{\frakp}} \text{ is a non-trivial norm on $E_{\frakp}^{\times}$}\},$$
which, by Lemma \ref{K*} and the construction of $N_{\Sigma_{E}}$,  has the property that $\Sigma_{E}^{(w)}\sqcup \Sigma_{E}^{(w)}c$ is the set of all primes $\frakp$ of $E$ above $p$. Hence $\Sigma_{E}^{(w)}$ is a  $p$-adic CM type of $E$, identified with  a set of embeddings $E\hookrightarrow \overline{\Q}_{p}$.


 
 \begin{lm}\label{split} Every prime of $F$ above $p$ splits in $E$, and  the CM type $\Sigma_{E}^{(w)}$ is a $p$-adic CM type of $E$. 
\end{lm}
\begin{proof} The first  assertion is part (2)  of Lemma \ref{K*}. The second assertion follows from part (3) of the same lemma, after noting that $\Sigma_{E}^{(w)}={\rm Inf}_{E/K^{*}}\Sigma^{*}(\iota_{p})$ for any $\iota_{p}\colon E\hookrightarrow \overline{\Q}_{p}$ inducing a prime in $\Sigma_E$.
\end{proof}

\subsubsection{$p$-adic Hecke characters}
Assume that $L$ splits $E$, and
let $\chi \colon E^{\times}\backslash E_{\A^{\infty}}^{\times}\to L^{\times}$ be a locally algebraic $p$-adic character. 
We say that $\chi$ is a Hecke character of $p$-adic infinity type $k \in \Z[\Hom(E,L)]$
if there exists an open subgroup $U \subset E_{\A^{\infty}}^{\times}$ 
such that 
$$
\chi(t)=\prod_{\sigma \in \Hom(E,L)}t_{\sigma}^{-k_{\sigma}}
$$
for $t \in U$. Note that this definition differs from the one in some of  the literature  (e.g. \cite{HT} and \cite{Hi3}) by a sign in the exponents; it agrees with \cite{BDP}.

If $\chi$ is a locally algebraic character of $p$-adic infinity type $k$ and  $\iota\colon L \rightarrow \C$ is an embedding, we can define the $\iota$-avatar $\chi^{\iota}\colon E_{\A}^{\times}\rightarrow \C^{\times}$ similarly to  \cite[Def. 1.5]{LZZ}. The embedding $\iota $ induces an isomorphism $\Hom(E, L)\to \Hom(E, \C)$ by $\sigma \mapsto \iota\circ \sigma $; the infinity type $k^{\iota}\in \Z[\Hom(E, \C)]$  of $\chi^{\iota} $ corresponds to $k$ under this bijection.
The association $\chi\to \chi^{\iota}$ defines   a bijection between locally algebraic Hecke characters over $E$ with values in $L$ and 
arithmetic Hecke characters over $E$ in the usual sense.  

\medskip
%

For example,  if $\lambda $ is as in \S\ref{sec: cm}, the character 
$$\lam^{(w)}(x):=\lambda(x)N_{\Sigma_{E}, p}^{(w)}(x_{p})^{-1} \colon   E^{\times}\backslash E_{\A^{\infty}}^{\times}\to K_{w}^{\times}$$
is a Hecke  character of $p$-adic infinity type  $\Sigma_{E}^{(w)}$.\footnote{Strictly speaking this assertion holds after considering $\lam^{(w)}$ as valued in some extension $L\supset K_{w}$ splitting $E$.} The place $w$ being fixed, in the rest of this paper we will often abuse notation by simply writing $\lambda$ in place of $\lambda^{(w)}$.

\subsubsection{Katz $p$-adic $L$-function} 
Let now $L$ be an extension of $K_{w}$ splitting $E$, and let $\Sigma_{E}$ be a $p$-adic CM type of $E$ over some $L$.
Let $E_{\infty}^{\sharp}$ be a finite extension of $E_{\infty}$ contained in $E^{\rm ab}$, and let $\Gamma^{\sharp}:=\Gal(E_{\infty}^{\sharp}/E)$. 
Let $\Lambda^\sharp:=\Z_p\llb \Gamma^\sharp\rrb_L$, $\Y^\sharp:= \Spec \Lambda^\sharp$.

By \cite{Ka, HT} and the first assertion of Lemma \ref{split}, there is an element 
$$L_{\Sigma_{E}}\in\Lambda^\sharp$$
 uniquely characterised by the interpolation property that we now describe.   
The domain of interpolation consists of 
locally algebraic $p$-adic Hecke characters
$\lam'\colon \Gamma^{\sharp}\to \overline{\Q}_{p}^{\times}$ with infinity type 
$$k\Sigma_{E}^{(w)}+ \kappa(1-c)$$
for $k\in \Z$,  $\kappa \in \Z[\Sigma_{E}^{(w)}]$ such that 
\begin{itemize}
\item[(i)] $k \geq 1$,  or 
\item[(ii)] $k \leq 1$ and $k\Sigma_{E}^{(w)}+ \kappa \in \Z_{> 0}[\Sigma_{E}^{(w)}]$. 
\end{itemize}
The interpolation property is then the  following. 
There exist explicit $p$-adic periods
$$\Omega_{\Sigma_{E}}=(\Omega_{\Sigma_{E},\frakp})_{\frakp\in \Sigma_{E}}\in (\bar{\Z}_{p}^{\times})^{\Sigma_{E}}$$
and, for each complex CM type $\Sigma_{E, \infty}$ of $E$, 
 complex periods  $$\Omega_{\Sigma_{E, \infty}}=(\Omega_{ \Sigma_{E, \infty},\tau})_{\tau\in\Sigma_{E, \infty}}\in (\C^{\times})^{\Sigma_{E, \infty}}$$
(both defined in \cite[(4.4)]{HT}) 
 such that  for any 
character $\lam'\colon \Gamma\to \overline{\Q}_{p}^{\times}$ in the domain of interpolation, 
and any $\iota\colon L(\lambda')\hookrightarrow \C$,
we have\footnote{Note that we are ignoring interpolation factors at places away from $p$ appearing elsewhere in the literature, since those, while non-integral, can be interpolated by polynomial functions on $\mathscr{Y}^{\sharp}$.}
\begin{align*}
\iota\left(\frac{L_{ \Sigma_{E}}(\lam')}{\Omega_{ \Sigma_{E}}^{k+2\kappa}}\right)=
e_p((\lam')^{-1})^{\iota})
\cdot {L^{(p)}(0,((\lam')^{-1})^{\iota}) \over \Omega_{\Sigma_{E,\iota}, }^{k+2\kappa}}
&\cdot \frac{\pi^{\kappa}\Gamma_\Sigma(k\Sigma_{E, \iota}+\kappa)}{(\Im \skewhf)^\kappa}\cdot {[\mathscr{O}_{E}^{\times}:\mathscr{O}_{F}^{\times}]\over \sqrt{|D_{F}|}}
\end{align*}

In the interpolation formula,  $\Sigma_{E,\iota}$ is the complex CM type induced from $\Sigma_{E}^{(w)}$ via $\iota$,  and we then identify $\kappa$ with $\kappa^{\iota} \in \Z_{\geq 0}[\Sigma_{E, \iota}]$; when $k\in \Z$ appears in the exponent of one of the periods $\Omega_{\Sigma}$ it is considered as $k\cdot\sum_{\tau\in \Sigma_{E}^{\iota}}\tau$.
If $\chi=(\lam')^{-1}$ is ramified at all $\frakp\vert p$, the $p$-Euler factor is given by 
$$
e_{p}(\chi^{\iota})=\prod_{\frakp \in \Sigma^{(w)}_{E}}e(\chi_{\frakp}^{\iota})
$$
for (dropping all superscripts $\iota$)
\beq\label{efrakp}
e(\chi_{\frakp}) =
{L(0, \chi_{\frakp}) \over
\varepsilon (0, \chi_{\frakp})L(1, \chi_{\frakp}^{-1})}  
\eeq
Finally, $\Gamma_{\Sigma}(k\Sigma_{E, \iota}+ \kappa)=\prod_{\tau \in \Sigma_{E, \iota}}\Gamma(k+\kappa_{\tau})$ 
for the usual $\Gamma$-function, 
and 
$\vartheta\in E $ as in \cite[\S3.1]{Hs1}.
All local epsilon factors in this paper are understood with respect to some  uniform  choice of additive characters of $E_{\frakp}$ of level one for all $\frakp\vert p$.

For a later consideration, we fix a sufficiently large extension $E^{\sharp}$  as above and consider the restriction of $L_{\Sigma_{E}}$ to certain open subsets of $\mathscr{Y}^{\sharp}$:
if $\lambda_{0}$ is a $p$-adic Hecke character factoring through $\Gamma^{\sharp}$ with values in $L$,  and $\Sigma_{E}$ is a $p$-adic CM type, we define
$$L_{\Sigma_{E}, \lambda_{0}}\in \Lambda$$
by 
$$L_{\Sigma_{E}, \lambda_{0}}(\chi'):= L_{\Sigma_{E}}(\lambda_{0}\chi').$$



We will henceforth drop the superscript $w$ from the notation for the character $\lambda^{(w)}$. 



%

\subsection{Factorisation of the Rankin--Selberg $p$-adic $L$-function} Let $\sigma:=\theta(\psi)$ as in \S\ref{sec: theta lift}. There is a factorisation
\begin{align}\label{factor}
L(s-1/2,\sigma_{E}\otimes (\chi_{0}\chi')^{-1})= L(s,\psi (\chi_{0}\chi')^{-1})L(s,\psi^{*}(\chi_{0}\chi')^{-1}))
\end{align} 
of complex (more precisely $K\otimes {\bf C}$-valued) $L$-functions, valid for algebraic Hecke characters $\chi'$ over $E$. It implies the following  factorisation  of $p$-adic $L$-functions.


\begin{lm} \label{p factor} Let $L_{p}(\sigma_{E}\otimes \chi_{0}^{-1})$ be the $p$-adic $L$-function associated with $\sigma=\theta(\psi)$ and the embedding $M\subset M_{v}\subset K_{w}$, where $w$ is the place of $K$ fixed above and $v$ its restriction to $M$. Let $L$ be a finite extension of $K_{w}$ splitting $E$.  We have
\begin{align}\label{eq fact}
L_{p}(\sigma_{E}\otimes \chi_{0}^{-1})\doteq  
  \frac{L_{\Sigma_{E},  \psi\chi_{0}^{*-1}}}{\Omega_{p,\Sigma_{E}}}\cdot
\frac{L_{\Sigma_{E}c,\psi^{*}\chi_{0}^{*-1}}}{\Omega_{p,\Sigma_{E}c}}
\end{align}
in $\Lambda=\Lambda_{L}$, where we use the symbol  $\doteq$ to signify an equality which holds up to multiplication by a constant in $\overline{\Q}^{\times}$.
\end{lm}
\begin{proof} 


We evaluate both sides of the proposed equality at   finite order characters $\chi'$ of $\Gamma$ which are sufficiently ramified in the sense that, for all primes $\frakp\vert p$ of $E$, $\chi_{\frakp}$ has conductor larger than the conductors of $\psi_{\frakp}$ and $\chi_{0, \frakp}$. 
This is sufficient as such characters are dense in  $\mathscr{Y}^{-}$ (cf. \cite[Lemma 10.2.1]{Di}).

Note first  the self-duality relation
 \begin{align}\label{SDrel}
\lambda'^{*}=\lambda'^{-1} |\cdot |^{-1}_{\A_{E}}
 \end{align} (where $|\cdot|=|\cdot|_{\A_{E}}$), valid for both $\lam'=\psi\chi_{0}^{-1}$ and $\lam'=\psi\chi_{0}^{*-1}$ (this follows from \eqref{cchar}). 
Evaluating  \eqref{factor}  at $s=1$, we find
$$L(1/2,\sigma_{E}\otimes (\chi_{0}\chi')^{-1})= L(0,\psi |\cdot| (\chi_{0}\chi')^{-1})L(0,\psi^{*}|\cdot| (\chi_{0}\chi')^{-1})=  
L(0, (\psi^{*}\chi_{0}^{*-1}\chi')^{-1})  L(0, ( \psi\chi_{0}^{*-1}   \chi')^{-1}).$$


Therefore the $L$-values agree with the ones being interpolated by the $p$-adic $L$-functions in \eqref{eq fact}. We now  compare the local Euler-like interpolation  factors and the complex periods.

 Recall from e.g. \cite[p. 119]{Hi8} that, for $\wp\mathscr{O}_{E}=\frakp\frakp^{c}$, we have
$$
\sigma_{\wp}\iso \pi(\psi_{\frakp},\psi_{\frakp^{c}})
$$ 
By the construction of $\Sigma_{E}^{(w)}$ and the Shimura--Taniyama formula for Hecke characters (see \cite[Proposition A.4.7.4 (ii)]{CCO}, cf. also the paragraph after Example A.4.8.3 \emph{ibid.}), the character $\psi_{\frakp}$ has values in $w$-adic units if and only if $\frakp\notin \Sigma_{E}^{(w)}$, equivalently $\frakp\in \Sigma_{E}^{(w)}c$.

Denoting by $\wp$ a fixed prime of $F$ above $p$ and by $\frakp$ the unique prime in $\Sigma_{E}^{(w)}$ above $\wp$, it follows that $\alpha_{\wp}=\psi_{\frakp^{c}}$ under the identification $E_{\frakp^{c}}=F_{\wp}$. Then, under our assumption on the ramification of $\chi'$, we have
\beq \label{epsig}
e_{\wp}(\sigma_{E}\otimes\chi_{0}^{-1}, \chi')= \varepsilon(0, \psi_{\frakp^{c}}^{-1}\chi_{0, \frakp}\chi'_{\frakp})\cdot
\varepsilon(0,\psi_{\frakp^{c}}^{-1}\chi_{0, \frakp^{c}}\chi'_{\frakp^{c}})
\eeq
whereas the  Katz interpolation factors  above $\wp$ are
$$\varepsilon(0,\psi_{\frakp}^{-1}\chi_{0, \frakp^{c}}\chi'_{\frakp}{}^{-1})^{-1}
\cdot \varepsilon(0,\psi_{\frakp}^{-1}\chi_{0, \frakp}\chi'_{\frakp^{c}}{}^{-1})^{-1}.
$$
By the functional equation
$$\varepsilon(s, \chi)^{-1}=\varepsilon (1-s, \chi^{-1})\chi(-1)$$
valid for any character $\chi$ of a local field (e.g. \cite[(23.4.2)]{BH}) and the self-dualities \eqref{SDrel}, these equal
\begin{gather*} 
\varepsilon(0,|\cdot|\psi_{\frakp}\chi_{0, \frakp^{c}}^{-1}\chi'_{\frakp}{})
\cdot \varepsilon(0,|\cdot|\psi_{\frakp}\chi_{0, \frakp}^{-1}\chi'_{\frakp^{c}})\\
=
\varepsilon(0,\psi_{\frakp^{c}}^{-1}\chi_{0, \frakp}\chi'_{\frakp}{})
\cdot \varepsilon(0,\psi_{\frakp^{c}}^{-1}\chi_{0, \frakp^{c}}\chi'_{\frakp^{c}})
\end{gather*}
matching \eqref{epsig}.



In regard to periods, we first note that the periods in the asserted equality are independent of $\chi'$. 
As 
$${\rm Ad}(\sigma)\iso \eta \oplus \Ind_{E}^{F}(\psi(\psi^{*})^{-1}),$$ we have 
$$
L(1,\sigma^{\iota}, {\rm ad})
=L(1,\eta)L(1,\psi(\psi^{*})^{-1})
=L(1,\eta)L(0,(\psi(\psi^{*})^{-1})^{D})
$$

The infinity type of $\iota(\psi(\psi^{*})^{-1})^{D}$ is $2\Sigma_{E}^{\iota}$.
From the algebraicity of Hecke $L$-values due to Shimura, we have
$$
L(0,(\psi(\psi^{*})^{-1})^{D}) \doteq \Omega_{\Sigma_{E}^{\iota}}^{2}
$$
(see \cite{Sh1} and \cite[\S4]{HT}, especially \cite[pp. 215-16]{HT}). 
Moreover, we have a period relation 
$$
\Omega_{\Sigma_{E}^{\iota}}^{1} \doteq \Omega_{\Sigma_{E}^{\iota}\circ c}^{1}
$$
(see \cite{Sh2} and \cite[Thm. 32.5]{Sh3}).

It follows that
$$\Omega_{\sigma}^{\iota}\doteq
\Omega_{\Sigma_{E}^{\iota}}^{2} \doteq \Omega_{\Sigma_{E}^{\iota}}\Omega_{\Sigma_{E}^{\iota} \circ c}.
$$
%
\end{proof}

\subsection{Construction of an auxiliary character} 
We will now look for a character $\chi_{0}$ suitable for our purposes. If a Hecke character $\lambda'$ of $E_{\A}^{\times}$ satisfies \eqref{SDrel}, then its functional equation relates $L(s,\lam')$ with $L(1-s, \lam'^{-1})=L(2-s, \lam'^{*}) =L(2-s, \lam')$; the sign of this functional equation is the root number
$$w(\lam'):=\varepsilon(1,  \lam')\in \{\pm 1\}.$$

We say that  finite-order character
  $\chi\colon  E^{\times}\backslash E^{\times}_{\A^{\infty}}\to \overline{\Q}^{\times}$  is \emph{anticyclotomic} if it satisfies the following two conditions, which are equivalent by \cite[Lemma 5.31]{HiH}:
 \begin{enumerate}
\item  $\chi^{*}=\chi^{-1}$;
\item there exists a 
finite-order character $\chi_{0} \colon  E^{\times}\backslash E^{\times}_{\A^{\infty}}\to \overline{\Q}^{\times}$ such that 
\beq
\label{chi chi0}\chi  = \chi_{0}/\chi_{0}^{*}.
\eeq
 \end{enumerate}

\begin{lm}\label{lm chi0}\label{twist}  Let $\lam$ be a Hecke character satisfying \eqref{SDrel}. Suppose that {the extension $E/F$ is ramified.}
Then there exist an anticyclotomic  finite-order character $\chi=\chi_{0}/\chi_{0}^{*}\colon E^{\times}\backslash E_{\bf A}^{\times}\to \overline{\Q}^{\times}$  such that the root number 
$$w( \lambda\chi)=+1.$$
\end{lm}
%
\begin{proof} 

%


Suppose first that $\lam$ satisfies the following condition. (We will later reduce to this case.)
$$\text{($*$) there is a  prime $\wp$ of $E$, ramified over $F$,
such that ${\rm ord}_{\wp}(\frak{C})$ is odd,}$$
where $\frak{C}$ is the conductor of $\lam$.
In particular, the norm ideal $N_{E/F}(\frak{C})$ is not a square. 
For a quadratic character $\chi'$ over $F$, let $\chi_{E}' = \chi' \circ N_{E/F}$ be the corresponding Hecke character over $E$. 
By definition, $\chi_{E}'$ is a quadratic Hecke character over $E$ and also anticyclotomic. 
For the latter, note that
$$
\chi_{E}'(a)\chi_{E}'^{*}(a)=\chi'(N_{E/F}(a))\chi'(N_{E/F}(c(a)))=\chi'(N_{E/F}(a))^{2}=1.
$$ 
We consider twists of $\lam$ by characters of the form $\chi_{E}'$ with the conductor of $\chi_{E}'$ prime to $\frak{C}$. 
Recall that the twist of a self-dual character by an anticyclotomic character is again self-dual. 
To prove existence of twist with change in the root number, it thus suffices to show that $\chi'$ can be chosen so that $\chi'(N_{E/F}(\frak{C}))$ takes value $1$ or $-1$. 
The sufficiency follows from the explicit root number formula for the twist $\lam\chi_{E}'$ in 
\cite[(3.4.6)]{Ta}. 
As the norm ideal $N_{E/F}(\frak{C})$ is not a square, the existence of desired $\chi'$ follows readily.

To reduce to condition ($*$), it suffices, given $\lambda$, to find ${\lambda}'=\lambda\chi_{1}/\chi_{1}^{*}$ satisfying ($*$). 
Let $\wp$ be a prime of $E$ ramified in $E/F$ and let $\delta:={\rm ord}_{\wp}(2)$.  Let $r$ be an odd integer greater than $\delta+2$ and  the exponent of $\wp$ in the conductor of $\lambda$. 
Let $\chi_{1}$ be a finite order character of $E_{\A}^{\times}/E^{\times}$ such that the exponent of $\wp$ in the conductor of $\chi_{1}$ 
is exactly $r-\delta$. 
As  $(1+\wp^{t-\delta} {\mathscr O}_{E})^{2}= 1+\wp^{t} {\mathscr O}_{E}$
for all $t\geq \delta+1$, the conductor  of $\chi^{2}_{1, \wp}$ is exactly $\wp^{r}$. 
Letting $\chi=\chi_{1}/\chi_{1}^{*}$, and denoting by $\varpi$ a uniformiser at $\wp$, for any $t\geq \delta+1$ we have for any $t\geq (r-\delta)/2$:
$$\chi_{\wp}(1+\varpi^{t}a )=\chi_{1,\wp}(1+\varpi^{t}a )\chi_{1,\wp}(1-\varpi^{t}a )^{-1}=\chi_{1,\wp}(1+\varpi^{t}a )^{2}.$$
It follows that the conductor of $\chi_{\wp}$ is the same as the conductor of $\chi_{1}^{2}$, that is, $\wp^{r}$;  by our choice of $r$ the same is true of $\lambda\chi$, and in particular ${\lambda}'$ satisfies ($*$).
%
\end{proof}

%
%
%
%
%
%

\begin{rema}\label{satisf} Let $\chi=\chi_{0}\chi_{0}^{*-1}$ be as in the Lemma and let $\psi:=\lambda\chi_{0}$, where $\lam$ is as fixed in the Introduction. Then $\psi$ satisfies  \eqref{cchar}; by the analogous formula to \eqref{factor}, $\sigma:=\theta(\psi)$ has root number $-1$ (i.e., it satisfies  \eqref{root n}).
\end{rema}


\subsection{$p$-adic Gross--Zagier   formula}
We recall a formula relating    Rankin--Selberg  $p$-adic $L$-functions  to points on abelian varieties; we will later deduce from it an  analogous formula for Katz $p$-adic $L$-functions. 
Let $\sigma$, $A=A_{\sigma}$, and $\chi_{0}$ be as in Theorem \ref{stdLp}, and assume that $E/F$ splits at each prime above $p$.  and that  
\beq\label{sign sigma}
\varepsilon(1/2,\sigma_{E}\otimes\chi_{0}^{-1})=-1.
\eeq


\begin{thm}\label{gzf} Let $\circ=\emptyset$ or $\circ=  \wp$ for a prime  $\wp\vert p$ of $F$. There is a `pair of points' 
$$\mathscr P_{\chi_{0}}\otimes \mathscr P^{\vee}_{\chi_{0}}\in A_{E}(\chi_{0}^{-1}\chi^{-1}_{\rm univ, \circ} ) \otimes_{\Lambda_{\circ}^{-}} A_{E}^{\vee}(\chi_{0}\chi_{\rm univ, \circ})\otimes_{\Lambda_{\circ}^{-}} {\mathscr{K}_{\circ}^{-}},$$ 
such that
$$\langle \mathscr P_{\chi_{0}}\otimes \mathscr P^{\vee}_{\chi_{0}}\rangle = L_{p}'(\sigma_{E}\otimes\chi_{0}^{-1})$$
in $\mathscr{K}_{\circ}^{-}$.
 Here $\langle x\otimes y  \rangle:=\langle x, y\rangle$ is the big height pairing relative to the cyclotomic logarithm as in \eqref{bigHP}.  
\end{thm}
\begin{proof} This follows from \cite[Theorem C.4]{Di}. Consider the  scheme $\Y_{V^p}/L$ corresponding to the rigid space with that name in \emph{loc. cit.} (in the sense that our $\Y_{V^p}$ is the spectrum of the ring of bounded functions on the space $\Y_{V^p} $ of \cite{Di}),
 for a choice of level $V^{p}\subset E^{\times}_{\A^{p\infty}}$; it parametrises continuous $p$-adic characters $\widetilde{\chi}$ of $E_{\A^{\infty}}^{\times}/E^{\times}V^{p}$ satisfying $\omega\widetilde{\chi}|_{F_{\A^{\infty}}^{\times}}=\one$. We denote by $\widetilde{\chi}_{\rm univ}\colon E_{\A^{\infty}}^{\times}/E^{\times}V^{p}\to \mathscr{O}(\mathscr{Y_{V^p}})^{\times}$ the universal character.  Assume that $\omega|_{V^{p}}=\one$. 
Then we may  identify ${\mathscr Y}^{-}$ with the connected component $\mathscr Y^{\circ }_{\chi_{0}}\subset \mathscr Y_{V^{p}}$ containing the character $\chi_{0}^{-1}$ via 
\beq\label{identify}
\chi\mapsto \widetilde{\chi}=\chi_{0}^{-1}\chi^{-1}.
\eeq

 Using the notation of \emph{loc. cit.} with the addition of a tilde,  the $p$-adic Gross--Zagier formula proved there has the form
$$\langle \widetilde{\mathscr P}^{+}(f^{+,p}) , \widetilde{\mathscr P}^{-}(f^{-,p})^{\iota} \rangle = \widetilde{L}_{p}'(\sigma_{E})\cdot \widetilde{\mathscr Q}(f^{+,p}, f^{-,p})|_{\mathscr Y^{}_{\chi_{0}}} \qquad \qquad \text{in $\Lambda^-$}$$
up to an explicit and nonzero rational constant. Here $\iota$ is the involution $\widetilde{\chi}\mapsto \widetilde{\chi}^{-1}$ on $\mathscr{Y}_{V^{p}}$,  and
the 
$$  \widetilde{\mathscr P}^{\pm}(f^{\pm,p})\in A^{\pm }((\widetilde{\chi}_{\rm univ})^{\pm 1})$$
 are families of Heegner points  associated with $E/F$ and   (the limits of   certain sequences of) parametrisations $f^{\pm, p}$  of $A$ and $A^{\vee}$ by  a (tower of) Shimura curves.
  The term $ \widetilde{\mathscr Q}(\cdot, \cdot)\in \Lambda^-$ is a product of local terms at primes not dividing  $p$. 
  
  By results of Tunnell and Saito explained in \cite[Introduction]{Di}, under the assumption \eqref{sign sigma} for each $\chi\in \mathscr{Y}_{\chi_{0}}$ we may find families of Shimura curve parametrisations  $f^{\pm, p}$ such that  $ \widetilde{\mathscr Q}(f^{+,p}, f^{-, p})(\chi)\neq 0$. Applying this result to a character $\chi$ in the image $ \mathscr{Y}_{\chi_{0}, \circ}$ of $\mathscr{Y}_{\circ}^{-}$ under the isomorphism $\mathscr{Y}^{-}\to \mathscr{Y}_{\chi_{0}}$, we find $f^{\pm, p}$ such that  $ \widetilde{\mathscr Q}(f^{+,p}, f^{-, p})|_{ \mathscr{Y}_{\chi_{0}, \circ}}\neq0$. 
Up to constants in $L^{\times}$, we have   $L_{p}'(\sigma_{E}\otimes \chi_{0}^{-1})(\chi) =- \widetilde{L}_{p}'(\sigma_{E})(\chi_{0}^{-1}\chi^{-1})$.
Then
we may choose, using the identification \eqref{identify}
$$\mathscr P_{\chi_{0}}\otimes \mathscr{P}_{\chi_{0}^{}}^{\vee}:=
- \widetilde{\mathscr Q}(f^{+,p}, f^{-,p})|^{-1}_{\mathscr{Y}_{\chi_{0}, \circ}}\cdot
\widetilde{\mathscr P}(f^{+})|_{\mathscr{Y}_{\chi_{0}, \circ}} \otimes\widetilde{\mathscr P}(f^{-})^{\iota}|_{\mathscr{Y}_{\chi_{0}, \circ}}.$$ 

There are four conditions to be verified in order to be able to invoke the result of \cite{Di}.\footnote{It is crucial here that in \cite{Di}  the sets of ramified primes of $E/F$ and of bad-reduction primes for $A$ are not required to be disjoint.} The first  one is (weaker than) the potential ordinariness of $A$, which can be verified after base-change to $E$ where it becomes the converse to part (1) of Lemma \ref{K*}. The second one is that all primes of $F$ above $p$ split in $E$: this is  satisfied by Lemma \ref{split}. Finally, the conditions on the central character and root number are satisfied by Remark \ref{satisf}.
\end{proof}

\subsection{Non-vanishing of $p$-adic $L$-functions} 


Let  $\lam$  be the character fixed in the Introduction.

\begin{thm}\label{nvt1} 
Let $\chi=\chi_{0}/\chi_{0}^{*}$ be as in Lemma \ref{twist}.
For every $\wp\vert p$, the restriction of the anticyclotomic  Katz $p$-adic L-function 
$$L_{\Sigma_{E},\lam\chi_{0}^{*}\chi_{0}^{-1}}|_{\mathscr{Y}^{-}_{\wp}}$$
does not vanish.
\end{thm}
\begin{proof}
By construction, $w(\lam\chi_{0}^{*}\chi_{0}^{-1})=+1$.
The theorem thus follows from \cite[Thm. B]{BuHi} combined with the main results of \cite{Hi3} and \cite{Hs1}. 
Here we only use the hypothesis $p \ndivide 2D_{F}$.
\end{proof}

\begin{thm}\label{nvt2} 
For every $\wp\vert p$, the restriction of the cyclotomic derivative 
$$
L_{\Sigma_{E}c, \lam^{*}}'|_{\mathscr{Y}^{-}_{{\wp}}}
$$
does not vanish.
\end{thm}
\begin{proof}
Recall from the Introduction that $\lam^{*}$ is self-dual with infinity type $\Sigma_{E}c$ and root number $-1$.
The theorem thus follows from \cite[Thm. C]{BuHi} combined with the main result of \cite{Bu}. 
Here we use the hypothesis $p \ndivide 2D_{F}h_{E}^{-}$.
\end{proof}


\subsection{Proofs of main theorems }
We introduce the useful category
$$\mathcal{CM}_{E,(K, \Sigma)},$$
described as follows. The  objects are abelian varieties $B$  over $E$ of dimension equal to  ${d\over 2}[K:\Q]$ for some $d\geq 1$,  together with an inclusion $i\colon R\hookrightarrow {\rm End}^{0}(B)$  of a  $K$-algebra $R$  of dimension $d$ such that the  type of $i$ is     $(K, d\Sigma)$. For two objects $B=(B, R,i)$, $B'=(B,  R',i')$   of $\mathcal{CM}_{E,(K, \Sigma)}$, let $R^{\circ }\subset R$ and $R^{\circ'}\subset R'$ be finite-index   subrings  in the integral closure of $\mathscr{O}_{K}$ in $R$, $R'$ whose image by $i$, $i'$ is contained in ${\rm End}(B)$, ${\rm End}(B')$ respectively. Let
$${\rm preHom}_{\mathcal{CM}_{E, (K, \Sigma)}}(B, B')$$
be the set of pairs of  morphisms $(f, \gamma)$ with  $f\colon B\to B'$, $\gamma\colon R\to R'$ such that $i'(\gamma(r))\circ f=f\circ i(r)$ for any $r\in R$. This is a module over a sufficiently small order $\mathscr{O} $ in $K$, and we let
$${\rm Hom}_{\mathcal{CM}_{E, (K, \Sigma)}}(B, B'):=
{\rm preHom}_{\mathcal{CM}_{E, (K, \Sigma)}}(B, B')\otimes_{\mathscr{O}}K.$$

If $(B, i, R)$ is an object of ${\mathcal{CM}_{E, (K, \Sigma)}}$ and $T$ is a finite-dimensional $K$-algebra, then Serre's construction provides a well-defined isomorphism class $(B\otimes_{K}T, i\otimes {\rm id}_{T}, R\otimes_{K}T)$  of objects in ${\mathcal{CM}_{E, (K, \Sigma)}}$, with action by the $K$-algebra $R\otimes_{K}T$.

Note that in $\mathcal{CM}_{E,(K, \Sigma)}$  any object $(B, i, R)$ is isomorphic to one such that ${\rm End}(B)$ contains the integral closure  $R^{\rm int}$ of $\OO_{K}$ in $R$: namely, if $R^{\circ}\subset R^{\rm int}$ is an order conatined in ${\rm End}(B)$ we may take $B':={\rm Hom}_{R^{\circ}}(R^{\rm int}, B)$. 
Given an  object $(B, i, R)$ of $\mathcal{CM}_{E, (K, \Sigma)}$ and a finite-order character $\chi\colon  \Gal(H_{\chi}/E)\to K^{\times}$   we define the twist 
$$B\otimes_{K}K(\chi)$$
 (an object of  $\mathcal{CM}_{E, (K, \Sigma)}$ with the  $R$-action of induced from the one on $B$) as follows. Assume  that $R\supset \OO_{K}$, which as noted above  is not restrictive. Then we may regard $\chi$ as an element in  $H^{1}(\Gal(H_{\chi}/E),  \OO_{K}^{\times})\subset  H^{1}(\Gal(H/E), {\rm Aut}_{K}(B))$; the abelian variety $B\otimes_{K} K(\chi)$ is the corresponding inner twist (denoted by $B^{\chi}$ in \cite{BDP}), so that for any finite character $\chi'\colon {\rm Gal}(\overline{E}/E)\to K^{\times}$ we have $(B\otimes K(\chi)(\chi')= B(\chi\chi')$.  

\medskip


Let us now return to our usual setting, so that $B=B_{\lambda}$. It is an object of ${\mathcal{CM}_{E, (K, \Sigma)}}$ with $R=K$.    (Note that, as the validity of the statements  we are interested in is invariant under $K$-linear isogenies, it is appropriate to work in this category.)
Let $\chi_{0}\colon {\rm Gal}(\overline{E}/E)\to K'^{\times}$ be a finite-order character, where $K'$ is the CM extension of  $K$ fixed above. Let $\psi:=\lambda \chi_{0}^{-1}$ and let $A=A_{\sigma}$ be the abelian variety associated with $\sigma=\theta(\psi) $ as in \ref{sec: theta lift}.
The abelian varieties $A_{\sigma}$ and  $B_{\psi, K'}$ have CM by $K'$. 

\begin{lm}\label{b and a}  There is an isomorphism in ${\mathcal{CM}_{E, (K', \Sigma)}}$
$$f\colon B_{\lambda}\otimes_{K} K'\to B_{\psi, K'}\otimes_{K'} K'(\chi_{0}^{-1})\cong A_{E}^{\oplus r} \otimes_{K} K'(\chi_{0}^{-1}).$$
\end{lm}
\begin{proof} The second isomorphism is \eqref{bpsi}. The  proof of the first one,  based on Casselman's theorem, is entirely analogous  to the proof of \cite[Lemma 2.9]{BDP}.
\end{proof}

%
%
%
%


\begin{proof}[Proof of Theorem \ref{MT2}] We prove the
$p$-adic Gross--Zagier formula of Theorem \ref{MT2}.
Let $\chi_{0}$ be as in Lemma \ref{twist}
 for $\lam':=\lam^{*}$,  let $A=A_{\theta(\lam\chi_{0}^{-1})}$ and let 
$$\mathscr P_{\chi_{0}}\otimes \mathscr P^{\vee}_{\chi_{0}}\in A_{E}(\chi_{0}^{-1}\chi^{-1}_{\rm univ, {\circ}} ) \otimes_{\Lambda_{{\circ}}^{-}} A_{E}^{\vee}(\chi_{0}\chi_{\rm univ, \circ})\otimes_{\Lambda_{{\circ}}^{-}} {\mathscr{K}_{{\circ}}^{-}}$$
  be as in Theorem \ref{gzf}.   Let $f$ be as in Lemma \ref{b and a} and let 
$$\mathscr{P}\otimes \mathscr{P}^{\vee}:={L_{\Sigma_{E}c,\lam^{*}\chi_{0}\chi_{0}^{*-1}}}|_{\Y_{{\circ}}^{-}}^{-1} \cdot
( f^{-1}\otimes f^{\vee})
(i_{1}(\mathscr{P}^{\vee}_{\chi_{0}}\otimes \mathscr P^{\vee}_{\chi_{0}}))|_{\Y_{{\circ}}^{-}},$$
which is an element of $B(\chi^{-1}_{\rm univ, {\circ}})_{\mathscr{K}_{{\circ}}^{-}}\otimes B(\chi_{\rm univ, {\circ}})_{\mathscr{K}_{{\circ}}^{-}}.$
 by Theorem \ref{nvt1}.

 
By the projection formula for heights \cite{MT}, for any $P_{1}\in B(\overline{E})$, $P_{2}\in B^{\vee}(\overline{E})$ we have 
$$\langle f^{-1}(P_{1}),  f^{\vee}(P_{2})\rangle_{B}=\langle P_{1}, P_{2}\rangle_{B_{\psi, K'}} = \langle P_{1}, P_{2}\rangle_{A_{E}^{\oplus r}}.$$
As maps of finitely generated  $\Lambda_{{\circ}}^{-}$-modules  are determined by their specialisations at finite order characters, this implies the analogous result for big height pairings. 
Then Theorem  \ref{MT2} follows from  Theorem \ref{gzf} and the factorisation \eqref{factor}.
\end{proof}

\subsubsection*{Proof of Theorem \ref{MT}}
We state and prove  the following slightly more precise version of Theorem \ref{MT}.  Recall from the Introduction that we say that a property $P$ holds \emph{for almost all} finite-order characters in $\mathscr{Y}^{-}_{{\circ}}$ if the set of those $\chi$ not satisfying $P$ is not Zariski dense in $\mathscr{Y}_{{\circ}}^{-}$. 

We keep the assumption of Theorem \ref{MT}.
\begin{thm}\label{MT-text} For almost all finite-order  $\chi\in \mathscr{Y}_{\circ}^{-}$, we have 
$$L_{\Sigma_{E}}'(\lam\chi)\neq 0,$$
the specialisation $\mathscr P \otimes \mathscr P^{\vee}(\chi)$ is  a well-defined and  non-zero element of $B(\chi^{-1})\otimes B^{\vee}(\chi)$, and
$$\langle \mathscr P\otimes \mathscr P^{\vee}(\chi)\rangle\neq 0.$$
\end{thm}
\begin{proof} 
The first assertion is equivalent to Theorem \ref{nvt2}. 
That the points are generically well-defined at $\chi$ amounts to the assertion that ${L_{\Sigma_{E},\lam\chi_{0}^{*}\chi_{0}^{-1}}}|_{\mathscr Y_{\wp}^{-}}\neq 0$, which is Theorem \ref{nvt1}. 
Finally the non-vanishing of $p$-adic heights follows from the other assertions and the Gross--Zagier formula of Theorem \ref{MT2}.
\end{proof}

\end{document}